\documentclass[11pt]{amsart}

\usepackage{amsmath,amscd,amsfonts,amsthm}
\usepackage[shortlabels]{enumitem}
\usepackage{hyperref}
\usepackage{graphicx}

\newcommand{\shrinkmargins}[1]{
  \addtolength{\textheight}{#1\topmargin}
  \addtolength{\textheight}{#1\topmargin}
  \addtolength{\textwidth}{#1\oddsidemargin}
  \addtolength{\textwidth}{#1\evensidemargin}
  \addtolength{\topmargin}{-#1\topmargin}
  \addtolength{\oddsidemargin}{-#1\oddsidemargin}
  \addtolength{\evensidemargin}{-#1\evensidemargin}
  }

\shrinkmargins{.7}

\DeclareMathOperator{\Conf}{Conf}
\DeclareMathOperator{\PConf}{PConf}

\DeclareMathOperator{\Emb}{Emb}

\DeclareMathOperator{\Top}{Top}

\newcommand{\field}[1]{\mathbb{#1}}

\newcommand{\Q}{\field{Q}}

\newcommand{\ra}{\rightarrow}

\newcommand{\beq}{\begin{displaymath}}
\newcommand{\eeq}{\end{displaymath}}
\newcommand{\beqn}{\begin{equation}}
\newcommand{\eeqn}{\end{equation}}
\newcommand{\Fin}{\mathcal{F}}
\newcommand{\FI}{\mathcal{FI}}
\newcommand{\FIsharp}{\FI\#}
\newcommand{\ncFin}{\mathcal{N}}
\newcommand{\Vect}{\mathcal{V}}

\theoremstyle{plain}
\newtheorem{thm}{Theorem}
\newtheorem{prop}[thm]{Proposition}
\newtheorem{cor}[thm]{Corollary}
\newtheorem{lem}[thm]{Lemma}

\theoremstyle{definition}
\newtheorem{defn}[thm]{Definition}

\newtheorem{exmp}[thm]{Example}

\theoremstyle{remark}
\newtheorem{rem}[thm]{Remark}
\newtheorem{remark}[thm]{Remark}

\title[Algebraic structure on configuration spaces of manifolds with flows]{Algebraic structures on cohomology of configuration spaces of manifolds with flows}
\author{Jordan S. Ellenberg}
\address{Department of Mathematics, University of Wisconsin, 480 Lincoln Drive, Madison, WI 53706}
\email{ellenber@math.wisc.edu}
\urladdr{http://www.math.wisc.edu/~ellenber/}
\author{John D. Wiltshire-Gordon}
\address{Department of Mathematics, University of Wisconsin, 480 Lincoln Drive, Madison, WI 53706}
\email{jwiltshiregordon@gmail.com}
\urladdr{https://sites.google.com/wisc.edu/jwg/home}

\begin{document}

\begin{abstract}
Let $\PConf^n M$ be the configuration space of ordered $n$-tuples of distinct points on a smooth manifold $M$ admitting a nowhere-vanishing vector field.  We show that the $i^{th}$ cohomology group with coefficients in a field $H^i(\PConf^n M, k)$ is an $\ncFin$-module, where $\ncFin$ is the category of noncommutative finite sets introduced by Pirashvili and Richter \cite{HochFunctorHomology}.  Studying the representation theory of $\ncFin$, we obtain new polynomiality results for the cohomology groups $H^i(\PConf^n M, k)$.  In the case of unordered configuration space $\Conf^n M = (\PConf^n M)/S_n$ and rational coefficients, we show that cohomology dimension is nondecreasing: $\dim H^i(\Conf^n M, \mathbb{Q}) \leq \dim H^i(\Conf^{n+1} M, \mathbb{Q})$.
\end{abstract}

\maketitle

If $M$ is a smooth manifold, we denote by $\Conf^n M$ the {\em configuration space} parametrizing unordered $n$-tuples $(p_1, \ldots, p_n)$ of distinct points in $M$.  By $\PConf^n M$, the {\em pure} or {\em ordered} configuration space, we mean the space parametrizing ordered $n$-tuples of distinct points; this space thus carries a free action of $S_n$, the quotient by which is $\Conf^n M$.

There is a large existing literature concerning homological stability for configuration spaces.  A foundational result of Arnol'd gives an explicit description of the cohomology of $\PConf^n \mathbb{R}^2$ \cite{ArnoldClassical}.  Cohen provides a similar result for $\PConf^n \mathbb{R}^m$ \cite{CohenClassical}.  The case of open manifolds $M$ goes back to work of McDuff and Segal \cite{McDuffClassical} \cite{SegalClassical}, who study the unordered case.  Later work of Church, Farb, Nagpal, and the first author~\cite{CEFN} provides a detailed account of symmetric group characters in cohomology in the ordered case.  They show that these characters are polynomial in a certain sense.  Their work relies on an $\FIsharp$-structure on $\PConf^n M$ that introduces points ``from infinity.''  

For closed manifolds, there may be topological obstructions to introducing points.  Nevertheless, in \cite{Church}, Church proves that characters appearing in the cohomology of configuration spaces with rational coefficients are \textit{eventually polynomial} using ``representation stability'' phenomena in the spirit of \cite{ChurchFarbRTHS}.  A subsequent proof of eventual polynomiality using the $\FI^{op}$-structure on $\PConf^n M$ appears in \cite{CEF}.  More traditional proofs have now appeared: see \cite{RWStability} and \cite{knudsen}.  

Another way to add points to a configuration on a closed manifold appears in recent work of Cantero-Palmer \cite{canteropalmer} and Berrick, Cohen, Wong, and Wu \cite{BCWW}; the idea is to introduce new points infinitesimally near old points.  These constructions rely on the existence of a nowhere-vanishing vector field on $M$.

The main contribution of the present paper is to use this geometric information associated with $M$ -- namely, the existence of a nowhere-vanishing vector field, or multiple such fields -- to improve \textit{eventual} polynomiality to \textit{immediate} polynomiality for some closed manifolds, via consideration of the natural module structures enjoyed by the cohomology of the configuration spaces of such a manifold.

We recall the  \textbf{category of non-commutative finite sets} $\ncFin$ introduced by Pirashvili and Richter in \cite{HochFunctorHomology}.  The objects of $\ncFin$ are finite sets;  the morphisms are finite set maps together with a linear ordering on each fiber.  Given two such morphisms,
\beq
\left(f : X \longrightarrow Y, \;\;\; \left( \leq_{f^{-1}(y_1)},\;\; \leq_{f^{-1}(y_2)},\;\; \ldots, \;\;\leq_{f^{-1}(y_n)} \right) \right) \phantom{,}
\eeq
\beq
\left(g : Y \longrightarrow Z, \;\;\; \left( \leq_{f^{-1}(z_1)},\;\; \leq_{f^{-1}(z_2)},\;\; \ldots, \;\;\leq_{f^{-1}(z_m)} \right) \right),
\eeq
their composite consists of the composite of the underlying finite set maps, together with induced total orders
\beq
\left(g \circ f : X \longrightarrow Z, \;\;\; \left( \leq_{(g \circ f)^{-1}(z_1)},\;\; \leq_{(g \circ f)^{-1}(z_2)},\;\; \ldots, \;\;\leq_{(g \circ f)^{-1}(z_m)} \right) \right),
\eeq
where, given $x, x' \in (g \circ f)^{-1}(z_p)$, we set $x \leq_{(g \circ f)^{-1}(z_p)} x'$ if $f(x) <_{(g)^{-1}(z_p)} f(x')$, or, when $f(x)=f(x')$, if $x \leq_{f^{-1}(f(x))} x'$.

If $\mathcal{C}$ is a category, a $\mathcal{C}$-module (or $\mathcal{C}$-representation) is a functor from $\mathcal{C}$ to $k$-vector spaces for some field $k$; morphisms of $\mathcal{C}$-modules are natural transformations.  We say a $\mathcal{C}$-module $V$ is {\em finitely generated} if there is a finite list of vectors $v_i \in Vc_i$ so that no proper $\mathcal{C}$-submodule of $V$ contains all the $v_i$.


We make certain simplifying assumptions about the manifold $M$:
\begin{equation}
\mbox{$M$ is smooth, orientable, connected, finite type, and $\mathrm{dim} M \geq 2$.} \label{eq:assume} \tag{$\ast$}
\end{equation}
\begin{thm} \label{thm:main} Let $M$ be a manifold satisfying \eqref{eq:assume} that admits a nowhere-vanishing vector field, and let $k$ be a coefficient field.  Then, for every $i$, the assignment $\{1, 2, \ldots, n \} \mapsto H^i(\PConf^n M, k)$ extends to a finitely generated $\ncFin$-module over the field $k$. 
\label{th:main}
\end{thm}

\begin{remark}  
Under the assumptions \eqref{eq:assume}, a nowhere-vanishing vector field 
exists if and only if the Euler characteristic vanishes (see, for example, \cite[p.552 Satz III]{vectorfields}).  Odd-dimensional manifolds have vanishing Euler characteristic by Poincar\'e duality.  Thus, still assuming \eqref{eq:assume}, any closed, 
odd-dimensional manifold with $\dim M \geq 3$ satisfies the conditions of Theorem~\ref{th:main}.   \end{remark}

\begin{remark}  We have a forgetful functor $\phi : \ncFin \longrightarrow \Fin$ to the more familiar category of finite sets that sends a map to its underlying set map.  We will see in Theorem \ref{thm:refinedmain} that when $M$ is a manifold admitting two vector fields which are everywhere linearly independent, the $\ncFin$-module structure on the cohomology of the pure configuration spaces of $M$ actually descends to an $\Fin$-module structure.
\end{remark}
\begin{remark}
We have another useful functor $\psi : \Delta \longrightarrow \ncFin$ from the category of nonempty finite linear orders and weakly monotone functions.  Indeed, any fiber of a monotone function inherits a natural ordering from the domain, and these orderings are compatible.  It follows that any $\ncFin$-representation can be restricted to a $\Delta$-representation ($\Delta$-representations are more commonly known as cosimplicial vector spaces).
\end{remark}

\begin{cor} \label{cor:dims} If $M$ satisfies \eqref{eq:assume} and admits a nowhere-vanishing vector field, then, for each $i$, there exists a polynomial $P_i$ such that
\begin{equation}
\dim H^i(\PConf^n M, k) = P_i(n)
\label{eq:agreement}
\end{equation}
for all $n > 0$.  The degree of $P_i$ is at most $i$ if $\dim M \geq 3$, and at most $2i$ if $\dim M =2$.
\label{co:polydim}
\end{cor}

\begin{remark}  That \eqref{eq:agreement} holds for all $n$ sufficiently large relative to $i$ is proved in \cite[Theorem 1.10]{CEFN}.  When $M$ is an open manifold, it is shown in \cite{CEF} that $H^i(\PConf^n M, k)$ carries a different structure, that of an {\em $\FIsharp$-module}; it follows that the identity \eqref{eq:agreement} holds for {\em all} non-negative $n$, zero included.  So the present result is of interest primarily for closed manifolds $M$.  In this case, \eqref{eq:agreement} need not hold for $n=0$; for instance, when $M = S^3$, we have that
\beq
\dim H^2(\PConf^n S^3, \Q) = (1/2)(n-1)(n-2)
\eeq
for all $n > 0$, but obviously not for $n=0$.  (See e.g. \cite[Theorem 3]{CT}.)  To see that some hypothesis on $M$ is necessary, note that
\beq
\dim H^1(\PConf^n S^2, \Q) = (1/2)(n^2 - 3n)
\eeq
for all $n \geq 3$, but this is evidently not the case for $n=1,2$.  
\label{rem:sharpness}
\end{remark}

In fact, the conclusion of Corollary~\ref{co:polydim} does not require the full $\ncFin$-module structure on the cohomology of configuration space; it uses only the fact that the sequence of spaces $H^i(\PConf^n M)$ forms a cosimplicial vector space, as follows immediately from the constructions of \cite{BCWW}, as we will explain in Remark \ref{rem:cosimplicialdimensions}.  However, the extra structure on cohomology allows us to control not only the dimension of the cohomology group, but its structure as a representation of $S_n$.  By a {\em character polynomial}, following the notation of \cite{CEF}, we mean a polynomial in the formal variables $X_1, X_2, \ldots$; a character polynomial can be interpreted as a character of $S_n$ for every $n$ by taking $X_j(\sigma)$ to be the number of $j$-cycles in the cycle decomposition of $\sigma$.  (In particular, $X_j$ is identically $0$ on $S_n$ when $n < j$.)  The variable $X_j$ is considered to have degree $j$.

\begin{cor}  \label{cor:characters} Let $M$ be a manifold satisfying \eqref{eq:assume} that admits a nowhere-vanishing vector field.  Then, for each $i$, there exists a character polynomial $P_i$ such that the character of $S_n$ acting on $H^i(\PConf^n M, \Q)$ agrees with $P_i$ for all $n > 0$.  The degree of $P_i$ is at most $i$ if $\dim M \geq 3$, and at most $2i$ if $\dim M =2$.
\label{co:polychar}
\end{cor}

We note that Corollary~\ref{co:polydim} (for rational coefficients) follows immediately from Corollary~\ref{co:polychar} by evaluating $P_i$ at $(X_1, X_2,\ldots) = (n,0,\ldots).$

As in Remark \ref{rem:sharpness}, the novelty of Corollary~\ref{co:polydim} lies in the uniformity of the bound for closed manifolds $M$; we already have from \cite[Theorem 1.8]{CEF} that the character of the representation $H^i(\PConf^n M, \Q)$ is given by a character polynomial for $n$ large enough relative to $i$, and for all $n$ when $M$ is open.

The following two corollaries concern rational cohomology of unordered configuration spaces on manifolds with a nowhere-vanishing vector field.  The first is a nondecreasing statement for Betti numbers as the number of points grows:
\begin{cor}
 Assume \eqref{eq:assume}, and suppose $M$ admits a nowhere-vanishing vector field.  Then, unordered configuration space $\Conf^n M$ satisfies
$$
\dim H^i(\Conf^n M, \Q) \leq \dim H^i(\Conf^{n+1} M, \Q).
$$
\label{co:monotone}
\end{cor}
The next corollary concerns the ``replication maps'' between configuration spaces defined by Cantero-Palmer \cite{canteropalmer}; this map replaces each of the $n$ points in a configuration with $m$ nearby points spaced out along a short ray in the direction of the nowhere-vanishing vector field. They show that these maps induce isomorphisms in cohomology groups.  We provide a new proof of \cite[Theorem B]{canteropalmer} in the special case of $\Q$ coefficients.

\begin{cor} \label{co:iso}
 Suppose $M$ satisfies \eqref{eq:assume} and admits a nowhere-vanishing vector field.  Then, for $n$ large enough relative to $i$ (if $\dim M = 2$, when $n \geq 2i$; otherwise when $n \geq i$) and any $m \geq 1$, the replication map $\Conf^{n} M \ra \Conf^{\, nm} M$ induces an isomorphism
$$
\dim H^i(\Conf^{\, nm} M, \mathbb{Q}) \overset{\sim}{\longrightarrow} \dim H^i(\Conf^{n} M, \mathbb{Q}).
$$
\label{co:cp}
\end{cor}

\section*{Acknowledgments}
The first author is supported by NSF Grant DMS-1402620 and a Guggenheim Fellowship.  The second author is supported by NSF Graduate Research Fellowship ID 2011127608 and the Algebra RTG at the University of Wisconsin, DMS-1502553.  The second author wishes to thank Aaron Mazel-Gee, Julian Rosen, and David Speyer for helpful conversations.  Both authors thank the hospitality of the Arizona Winter School, where this paper was initiated.

\section{Representation theory of $\ncFin$}

As we develop the representation theory of $\ncFin$, we will be able to deduce Corollaries \ref{cor:dims}, \ref{cor:characters}, \ref{co:monotone}, and \ref{co:cp} from Theorem \ref{thm:main}.  The proof of Theorem \ref{thm:main} appears in \S \ref{sect:mainproof}.

The \textbf{simplex category} $\Delta$ is the category of finite, non-empty, totally-ordered sets with weakly monotone maps.  We use the symbol $[n] = \{1, \ldots, n\}$ to denote the usual ordered set with $n$ elements.  Lemma~\ref{lem:dkrl} below is an elementary consequence of the classical Dold-Kan correspondence; see \cite{DoldClassical} or \cite[Section 8.4]{weibel} for a modern treatment.

Write $\Vect$ for the category of finite dimensional vector spaces over some field $k$.  The Dold-Kan correspondence gives an equivalence of categories between the functor category $\Vect^{\Delta}$ (the category of ``cosimplicial vector spaces'') and the category of $\Vect$\nobreak\mbox{-}\nobreak\hspace{0pt}cochain complexes supported in non-negative degree.  Under the correspondence, the cochain complex $k[-p]$ consisting of a single 1-dimensional vector space in cohomological degree $p$ becomes an Eilenberg--MacLane representation $K(k,p)$ satisfying $\dim K(k, p)[n] = \binom{n-1}{p}$.

\begin{lem}[Dold-Kan Representation Lemma] \label{lem:dkrl}
Let $V : \Delta \longrightarrow \Vect$ be a $\Delta$-representation over $k$.  The following conditions are equivalent:
\begin{enumerate}[(a)]
\item $V$ is finitely generated;
\item the sequence $\dim V[n]$ is bounded above by a polynomial in $n$;
\item the cochain complex corresponding to $V$ under Dold-Kan has finite support;
\item $V$ has finite length;
\item the sequence $\dim V[n]$ coincides with a polynomial in $n$.
\end{enumerate}
\end{lem}
\begin{proof}
The cardinality of $\Delta([q], [n])$ (the set of morphisms from $[q]$ to $[n]$ in $\Delta$) is bounded above by the cardinality of $\Fin([q], [n])$, so a basic projective $k \cdot \Delta([q], -)$ has dimension sequence bounded above by $n^q$, and $(a) \implies (b)$.  (The functor $k \cdot \Delta([q], -)$ denotes the composite of the representable functor $[n] \mapsto \Delta([q],[n])$ and the functor sending a finite set to the free $k$-vector space on that set.)  If the cochain complex corresponding to $V$ has $k[-p]$ as a simple constituent, then the dimension sequence of $V$ can have no polynomial bound of degree less than $p$.  It follows that unbounded cochain complexes correspond to $\Delta$\nobreak\mbox{-}\nobreak\hspace{0pt}representations with degree sequences that grow faster than any polynomial, and so $(b) \implies (c)$.  Finitely-supported cochain complexes have finite length, so $V$ must have polynomial dimension sequence; in short, $(c) \implies (d) \land (e)$.  Since $(d) \implies (a)$ and $(e) \implies (b)$, we are done.
\end{proof}

\begin{proof}[Proof of Corollary \ref{cor:dims} on polynomiality of dimensions]
For any $p \in \mathbb{N}$, the dimension sequence of a basic projective $\ncFin$-representation $k \cdot \ncFin([p],-)$ is bounded above by a polynomial, so the Dold-Kan representation Lemma \ref{lem:dkrl} gives that it is finite length when restricted to $\Delta$.  Any finitely generated $\ncFin$-representation is a quotient of a finite sum of these basic projectives; in particular, the $\ncFin$-representation $H^i(\PConf^n M, k)$ is such a quotient by Theorem~\ref{thm:main}.  Reapplying Lemma \ref{lem:dkrl} proves Corollary \ref{cor:dims}.  For the bound on degree, we recall from \cite[Th 6.3.1, Rem 6.3.3]{CEF} that $H^i(\PConf^n M, k)$ is bounded above by a polynomial of degree $i$ when $\dim M \geq 3$, or $2i$, if $\dim M = 2$.  The proof there has a running assumption that $k$ has characteristic $0$, but if only an upper bound is desired, that condition is superfluous; the proof uses a filtration of $H^i(\PConf^n M, k)$ by the $E^\infty$ page of a spectral sequence constructed by Totaro, whose entries are shown in the proof of \cite[Th 6.3.1]{CEF} to have dimension bounded above by a polynomial of degree $i$ (resp. $2i$), and this spectral sequence doesn't require rational coefficients.
\end{proof}

\begin{thm} \label{thm:ncfgfl}
An $\ncFin$-representation is finitely generated if and only if it is finite length.
\end{thm}
\begin{proof}
We will show that every finitely generated $\ncFin$-representation restricts to a finite length $\Delta$-representation; since any nontrivial inclusion of $\ncFin$-representations is in particular a nontrivial inclusion of $\Delta$-representations, the theorem follows.  Every finitely generated $\ncFin$\nobreak\mbox{-}\nobreak\hspace{0pt}representation is a quotient of a finite sum of representations of the form $k \cdot \ncFin([p],-)$ for various $p \in \mathbb{N}$, so it suffices to prove the theorem on these basic projectives.  As before, the dimension of $k \cdot \ncFin([p],[n])$ is bounded above by a polynomial in $n$; the claim then follows from the Dold-Kan representation Lemma~\ref{lem:dkrl}.
\end{proof}
\begin{thm} \label{thm:descend}
If $W$ is a simple $\Fin$-representation, then $W \circ \phi$ is a simple $\ncFin$-representation.  Further, every simple $\ncFin$-representation arises in this way.
\end{thm}
\begin{proof}
The first claim is clear, since the functor $\phi$ is full and surjective on objects.  In the other direction, we must show that any two $\ncFin$-morphisms that coincide after applying $\phi$ induce the same map under any simple $\ncFin$-representation.

Given a simple $\ncFin$-representation $V$, let $m$ be the smallest integer so that $V[m] \neq 0$, and pick some non-zero vector $v \in V[m]$.  Since $V$ is simple, $v$ generates the entire representation.  In particular, the vectors $(Vf)(v)$ with $f \in \ncFin([m], [n])$ span $V[n]$.  Given any two morphisms $p, q \in \ncFin([n], [n'])$ such that $\phi(p)=\phi(q)$, we may check that the induced maps $Vp$ and $Vq$ are equal by checking that they act the same way on the spanning vectors $(Vf)(v)$.

Note that $\phi(p \circ f) = \phi(p) \circ \phi(f) = \phi(q) \circ \phi(f) = \phi(q \circ f)$.  If $\phi(p \circ f)$ is an injection, then $p \circ f = q \circ f$ since injections in $\Fin$ have unique lifts to $\ncFin$.  In this case, it follows that $V(p \circ f) = V(q \circ f)$.  If $\phi(p \circ f)$ is not an injection, then both maps $V(p \circ f)$ and $V(q \circ f)$ factor through some $V[r]$ with $r < m$.  But $V[r] = 0$ since $m$ is minimal, and so $V(p \circ f) = 0 = V(q \circ f)$.  Thus, in either case $V(p \circ f) = V(q \circ f)$.  It follows that $V(p)(V(f)(v)) = V(q)(V(f)(v))$, and the claim is proved.
\end{proof}

\begin{cor} \label{cor:charpoly}
If the field $k$ has characteristic zero, the character of any finitely generated $\ncFin$-representation $V$ is a polynomial in the cycle-counts $X_1, X_2, \ldots$ for all $n > 0$.
\end{cor}
\begin{proof}
By the previous Theorems \ref{thm:ncfgfl} and \ref{thm:descend}, it suffices to check the property on the simple $\Fin$-representations.  Rains \cite[Theorem 4.1]{rains} deduces this property from work of Putcha \cite{fulltransformationmonoid} on the representations of the endomorphism monoid of a finite set.  A separate proof appears in \cite[\S6.7.3]{UniformlyPresentedVectorSpaces}.
\end{proof}
\begin{proof}[Proof of Corollary \ref{cor:characters} on polynomiality of characters]
By Theorem \ref{thm:main}, the $\ncFin$-representation $H^i(\PConf^n(M), \Q)$ is finitely generated.  It follows from Corollary \ref{cor:charpoly} that its character is given by a polynomial in the cycle-counts.  The bound on degree is proved as in the proof of Corollary~\ref{cor:dims}.
\end{proof}

\begin{proof}[Proof of Corollaries ~\ref{co:monotone} and \ref{co:cp} on cohomology of unordered configuration spaces]

Working with rational coefficients, $H^i(\Conf^n M, \mathbb{Q}) = H^i(\PConf^n M, \mathbb{Q})^{S_n}$.  Following the notation of \cite{UniformlyPresentedVectorSpaces}, the only simple $\ncFin$-representations in which trivial $S_n$-representations appear are those denoted $D_0, D_1, C_2, C_3, C_4, \ldots$, which we now describe explicitly.

A basis for $C_k[n]$ is given by all $k$-element subsets $ S \subseteq [n]$.  An $\ncFin$-arrow $f : [n] \longrightarrow [n']$ acts by $S \mapsto \phi(f)(S)$, provided $\phi(f)(S)$ still has $k$ elements; if $\phi(f)(S)$ has fewer than $k$ elements, then $S \mapsto 0$.  In particular, $\dim C_k^{S_n}$ is $1$ when $n \geq k$ and $0$ when $n < k$.  The simple $D_1$ is one-dimensional at every object except $[0]$, with arrows acting by the one-by-one matrix $(1)$; $D_0$ is one-dimensional at $[0]$ and zero everywhere else.  

From this description, we see that every simple $\ncFin$-representation $V$ defined over $\mathbb{Q}$ satisfies $0 \leq \dim V[n]^{S_n} \leq \dim V[n+1]^{S_{n+1}} \leq 1$ for $n \geq 1$, and so $\dim H^i(\Conf^n M, \mathbb{Q})^{S_n} \leq \dim H^i(\Conf^{n+1} M, \mathbb{Q})^{S_{n+1}}$.  This proves Corollary~\ref{co:monotone}.

In order to deduce the claim about replication maps, note that the simples appearing in the $\ncFin$-representation $H^i(\Conf^n M, \mathbb{Q})$ cannot grow faster than a polynomial of degree $d=2i$ ($d=i$ if the manifold has dimension at least $3$) by the bounds on $\dim H^i(\Conf^n M, \mathbb{Q})$ given in \cite[Theorem 6.3.1 and Remark 6.3.3]{CEF}.  Since $n \mapsto \dim C_k[n]$, $n \geq 1$, is a polynomial of degree $k$, we are only concerned with the simples $D_1, C_2, C_3, \ldots, C_{d}$.

Let $V$ be one of these simples.  Given a finite set map $f : [n] \longrightarrow [n']$ whose image has size at least $d \geq 1$, we show that the composite $\overline{f} : V[n]^{S_n} \longrightarrow V[n] \overset{f_*}{\longrightarrow} V[n'] \longrightarrow V[n']_{S_{n'}} \overset{\sim}{\longrightarrow} V[n']^{S_{n'}}$ is an isomorphism.  The case of $V=D_1$ is immediate.  If $V = C_k$ with $k \leq d$, then pick a $k$-element subset of $[n]$ so that $f$ restricts to an injection on these elements.  It follows that this subset maps via $f_*$ to a nonzero vector, and that $\overline{f}$ maps the average of all $k$-element subsets to a nonzero vector.  Recalling that the source and target of $\overline{f}$ are at most one-dimensional, we are done.

It follows that the map $\overline{f}$ is an isomorphism for any simple whose dimension grows as a polynomial of degree at most $d$, and hence for any finite extension of such simples, by the $5$-lemma.  The claim follows for the finite length $\ncFin$-representation  $H^i(\Conf^n M, \mathbb{Q})$.  This proves Corollary~\ref{co:cp}.
\end{proof}

\section{The space of embedded disks} \label{sec:emb}
One convenient proof of Theorem \ref{thm:main} uses the vector field to flow points apart from one another.  Up to homotopy, all the choices involved in this construction become irrelevant; this gives a homotopy action of $\ncFin^{op}$ on $\PConf$.  However, we aim to prove Theorem \ref{thm:main} by working with real spaces, and not just homotopy types.  As a consequence, the proof retains more geometry, and provides a stronger result, Theorem \ref{thm:refinedmain}.

In furtherance of this plan, we build a more-flexible sequence of spaces $X_n \simeq \PConf^nM$ that are homotopy equivalent to configuration spaces but allow more structure maps.  The spaces $X_n$ will be spaces of embedded disks, and the new structure maps will come from restricting along embedded disks within disks.  The convenient proof about flowing points can be recovered by imagining tangentially-embedded closed $1$-disks---integral curves for the vector field.
\subsection{Set-up}
Write $\mathbb{D}^l$ for the closed unit ball in $\mathbb{R}^l$, which is a smooth manifold with boundary, and $\mathbb{D}^l_n = \mathbb{D}^l \times [n]$, for the disjoint union of $n$ copies of this manifold.  Write $C^\infty(\mathbb{D}^l_n, M)$ for the set of smooth maps $\mathbb{D}^l_n \to M$, equipped with the ``smooth compact-open topology,'' which is also known as the weak $C^\infty$ topology, but in this case coincides with the strong topology, since $\mathbb{D}^l_n$ is compact.

\subsection{Properties of the weak topology}
We shall need several properties of the weak topology; a good reference is the textbook \cite{Hirsch}.  As a sanity check, $C^\infty([n], M) \cong M^n$, and in particular, $C^\infty([1], M) \cong M$.  For any triple of smooth manifolds $M_1, M_2, M_3$, the composition map
$$
C^\infty(M_1, M_2) \times C^\infty(M_2, M_3) \to C^\infty(M_1, M_3)
$$
is continuous \cite[\S2.2.4~ex.~10]{Hirsch}.  This fact enriches the category of smooth maps in topological spaces.  Also, setting $M_1 = [1]$ to be a single point, we find that evaluation of a smooth map at a point is jointly continuous in the map and the point.  Similarly, differentiation induces a continuous map
$$
C^\infty(M_1, M_2) \to C^\infty(TM_1, TM_2),
$$
guaranteeing that differentiation is continuous in the map and the point.

We shall also need the fact that $C^\infty(\mathbb{D}^l_n, M)$ is a metric space, hence paracompact \cite[Thm. 2.4.4 and subsequent discussion]{Hirsch}, and that the subset of embeddings, written $\Emb(M_1, M_2) \subseteq C^\infty(M_1, M_2)$, is open \cite[Thm. 2.1.4]{Hirsch}.  Since the composite of two embeddings is again an embedding, this makes $\Emb$ into a $\Top$-enriched category as well.  Finally, we make frequent implicit use of the partial evaluation map
\begin{equation} \label{equation:partial_evaluation}
C^\infty(M_1 \times M_2, M_3) \to C^0(M_1, C^\infty(M_2, M_3)),
\end{equation}
where $C^0$ denotes continuous maps between topological spaces in the compact-open topology.  This fact lets us specify continuous maps by giving a formula.  For example, we show continuity of the straight-line path from the identity map on $\mathbb{D}^l_n$ to the central map $(x, i) \mapsto (0,i)$.

\begin{exmp}
The map $\sigma \colon [0,1] \to C^\infty(\mathbb{D}^l_n, \mathbb{D}^l_n)$ given by
$\sigma(t) =  \big((x, i) \mapsto (tx, i) \big)$ is continuous.  To see why, consider the map $(t, x, i) \mapsto (tx, i)$, which is plainly smooth, and so provides an element of $C^\infty([0,1] \times \mathbb{D}^l_n, \mathbb{D}^l_n)$.  Applying \eqref{equation:partial_evaluation}, we find $\sigma \in C^0([0,1], C^\infty(\mathbb{D}^l_n, \mathbb{D}^l_n))$, and so $\sigma$ is continuous.
\end{exmp}
\subsection{The space of embeddings adapted to smoothly-chosen $l$-frame} \label{sec:adapted}
Suppose $M$ carries $l$ everywhere linearly independent smooth vector fields, encoded as a map from a trivial bundle
$$
f \colon M \times \mathbb{R}^l \to TM
$$
so that, for every point $p \in M$, the linear map $f(p) \colon \mathbb{R}^l \to T_pM$ is an injection.  Suppose $\varphi \colon \mathbb{D}^l_n \to M$ is an embedding, and that $(x, i) \in \mathbb{D}^l_n$ is a point, and write $p = \varphi(x, i)$.

From this set-up, there are two $l$-frames at $p$ that arise naturally: $f(p)$, and the derivative $D\varphi(p) \colon \mathbb{R}^l \to T_pM$, which is an injection since $\varphi$ is an embedding.  The next definition asks that every convex combination of these two frames also have full rank.

\begin{defn} \label{defn:adapted}
A map $\varphi \colon \mathbb{D}^l_n \to M$ is \textbf{adapted to $f$ at $(x, i)$} if the linear map
$$
t \cdot D\varphi(x, i) \, + \, (1-t) \cdot f(\varphi(x, i))
$$ 
has full rank for all $t \in [0, 1]$.  We say, simply, ``$\varphi$ is \textbf{adapted to $f$}'' if it is adapted to $f$ at every point $(x, i) \in \mathbb{D}^l_n$.  Write
$$
\Emb^f(\mathbb{D}^l_n, M) \subseteq C^\infty(\mathbb{D}^l_n, M)
$$
for the subset of embeddings adapted to $f$.
\end{defn}
For example, if $l=1$ and $\varphi$ is an integral curve for the vector field $f$, then the derivative of $\varphi$ matches $f$ exactly, and so $\varphi$ is adapted to $f$ trivially.  However, the definition is more flexible: as long as the derivative maintains a direct line of sight to the frame, all is well.

The goal of \S\ref{sec:emb}, achieved in Theorem \ref{thm:equiv}, is to show that the central evaluation map
$$
\begin{array}{ccc}
\zeta \colon \Emb^f(\mathbb{D}^l_n, M) & \to & \PConf^n(M) \\
\varphi & \mapsto & \big(\varphi(0,1), \, \ldots, \, \varphi(0,n)\big)
\end{array}
$$
is a homotopy equivalence.

\begin{rem}
It would be simpler to require that $\varphi$ be exactly tangent to the frame at every point, and indeed, this approach works well for $l=1$.  However, for $l \geq 2$, there may be pairs of vector fields that are everywhere linearly independent, but which do not admit integral surfaces.  In such cases, the space of embeddings is empty, and therefore unlikely to be homotopy equivalent to configuration space $\PConf^n M$.
\end{rem}

\begin{prop} \label{prop:adapted_open}
The subset $\Emb^f(\mathbb{D}^l_n, M)$ is open in $C^\infty(\mathbb{D}^l_n, M)$.  
\end{prop}
\begin{proof}
In the $C^\infty$ topology, the differentiation and evaluation maps
$$
\begin{array}{ccc}
(\varphi, x, i) \mapsto D\varphi(x, i) & \mbox{and} & (\varphi, x, i) \mapsto \varphi(x, i)
\end{array}
$$
are continuous, and consequently the map
$$
(t, x, i, \varphi) \mapsto t \cdot D\varphi(x, i) \, + \, (1-t) \cdot f(\varphi(x, i))
$$
is continuous as well.  Consider the preimage of the open subset of full-rank matrices.  If $\varphi$ is adapted to $f$, then this preimage contains the slice $[0,1] \times \mathbb{D}^l_n \times \{\varphi\}$.  Since $[0,1] \times \mathbb{D}^l_n$ is compact, the tube lemma provides a $C^\infty$ open neighborhood of $\varphi$ that consists of embeddings adapted to $f$.
\end{proof}

\subsection{The category $\mathcal{F}E_l$ and its functor to the category of embeddings}
Define a category $\mathcal{F}E_l$ with objects $[n]$ for $n \in \mathbb{N}$, and a topological space of morphisms $[n] \to [n']$ corresponding to embeddings $\mathbb{D}^l_n \to \mathbb{D}^l_{n'}$ that restrict to a composite of positive scalings and translations on every $\mathbb{D}^l \times \{i\}$.
Equivalently, $\mathcal{F}E_l([n], [1])$ is the space of embeddings $\mathbb{D}^l_n \to \mathbb{D}^l_1$ of the form $(x, i) \mapsto a_ix +b_i$ for various $a_i \in \mathbb{R}_{>0}$ and $b_i \in \mathbb{D}^l$, and a general morphism of $\mathcal{F}E_l$ is a disjoint union of $n'$ such maps.

Consequently, and using \eqref{equation:partial_evaluation}, there is a continuous map
$$
\mathcal{F}E_l([n], [n']) \to \Emb(\mathbb{D}^l_n, \mathbb{D}^l_{n'})
$$
for each $n, n' \in \mathbb{N}$, and this map respects composition.  In other words, there is a $\Top$-enriched functor $\xi \colon \mathcal{F}E_l \to \Emb$ given by $\xi[n] = \mathbb{D}^l_n$.

\begin{prop} \label{prop:action}
The assignment $[n] \mapsto \Emb^f(\mathbb{D}^l_n, M)$ extends to a functor $\mathcal{F}E_l^{op} \to \mathrm{Top}$ where
$$
\mathcal{F}E_l([n], [n']) \times \Emb^f(\mathbb{D}^l_{n'}, M) \mapsto \Emb^f(\mathbb{D}^l_n, M)
$$
is given by precomposition.
\end{prop}
\begin{proof}
Use $\xi$, and note that positive rescaling preserves being adapted to $f$.
\end{proof}

\subsection{Finding a subdisk that is adapted to $f$}.
We argue that any map adapted to $f$ at the center points $(0,1), \ldots, (0,n) \in \mathbb{D}^l_n$ has a small neighborhood of these points where the map is an embedding that is adapted to $f$.  We will need a standard lemma.
\begin{lem} \label{lem:embedding_neighborhood} 
Let $N, M$ be smooth manifolds, and suppose that $K \subseteq N$ is a compact subset.  If $\phi \colon N \to M$ is a smooth map so that $\phi |_K$ is an injection and $D\phi(x)$ is a linear injection for all $x \in K$, then $K$ has an open neighborhood $V \subseteq N$ so that $\phi|_V$ is an embedding.
\end{lem}
\begin{proof}
See, for example, \cite[2.1 Ex. 7]{Hirsch} or \cite{MSE1782385}.
\end{proof}
The next result takes $K = \{(0,1), \ldots, (0,n)\} \subseteq \mathbb{D}^l_n$ to be the centers of the $l$-disks.
\begin{prop} \label{prop:get_alpha}
If $\varphi \colon \mathbb{D}^l_n \to M$ is a smooth map that is adapted to $f$ at every $(0, i) \in K$, and if $\varphi|_K$ is an injection, then for sufficiently-small $\alpha \in (0,1]$, the map $(x, i) \mapsto \varphi(\alpha \cdot x, i)$ is an embedding that is adapted to $f$.
\end{prop}
\begin{proof}
Since $\varphi$ is adapted to $f$ at each point $(0,i)$, the matrices $t \cdot D\varphi(0, i) + (1-t) \cdot f(\varphi(0, i))$ have full rank for every $t \in [0,1]$.  Setting $t = 1$, we see that $D\varphi(0,i)$ has full rank for all $i \in [n]$.  By Lemma \ref{lem:embedding_neighborhood}, we may pick $\alpha_0$ so that $(x, i) \mapsto \varphi(\alpha_0 \cdot x, i)$ is an embedding.  Since
$$
(x, i, t) \mapsto t \cdot D\varphi(\alpha_0 x, i) \, + \, (1-t) \cdot f(\varphi(\alpha_0 \cdot x, i))
$$
is continuous, the preimage of the full-rank matrices is open.  We claim that this preimage contains the slice $\{0\} \times [n] \times [0,1]$.  Assuming the claim, the tube lemma supplies $\alpha_1 \in (0,1]$ so that $(\alpha_1 \cdot \mathbb{D}^l_n) \times [0,1]$ also maps to full-rank matrices.  Setting $\alpha = \alpha_1 \alpha_0$ completes the proof.

To prove the claim, we show that for any $i \in [n]$ and $t \in [0,1]$,
$$
t \cdot \alpha_0 D\varphi(0,i) + (1-t) \cdot f(\varphi(0,i))
$$
is a non-zero scalar multiple of a matrix that we have already assumed has full rank.  If $t \in \{0,1\}$, the claim is clear.  If $t \in (0,1)$, then $t \cdot \alpha_0 > 0$ and $1-t > 0$, and so we may divide by the sum of these numbers to obtain a convex combination of $D\varphi(0,i)$ and $f(\varphi(0,i))$.  Every such convex combination is full rank by assumption.
\end{proof}

\subsection{Riemannian structure, the exponential map, and tiny embeddings}
Fix a Riemannian structure on $M$, and write $\exp \colon TM \supseteq U \to M$ for the exponential map, where $U$ is some fiberwise convex open neighborhood of the zero section.  Pick $U$ small enough so
$$
(\pi_M, \exp) \colon U \to M \times M
$$
is a diffeomorphism onto an open set $L \subseteq M \times M$ containing the diagonal.  If $(p, p') \in L$, we write $\log_p(p') = (\pi_m, \exp)^{-1}(p, p')$.  Since $U$ was chosen to be fiberwise convex, if $p', p'' \in M$ are both in the domain of $\log_p$, then any convex combination
\begin{equation} \label{equation:convex}
t \cdot \log_p(p') \, + \, (1-t) \cdot \log_p(p'')\;\;\;\;\;\;\; \mbox{ for $t \in [0, 1]$}
\end{equation}
is in the domain of the exponential map.
\begin{defn}
An embedding $\varphi \colon \mathbb{D}^l_n \to M$ is \textbf{tiny} if $(\varphi(0, i), \varphi(x, i)) \in L$ for all $(x, i) \in \mathbb{D}^l_n$. 
\end{defn}
\noindent
If $\varphi$ is tiny, then $\log_{\varphi(0,i)} \varphi(x, i) \in T_{\varphi(0, i)}M$ exists and depends smoothly on $(x, i) \in \mathbb{D}^l_n$.
\begin{prop} \label{prop:tiny_open}
The tiny embeddings form an open subset of $C^\infty(\mathbb{D}^l_n, M)$.  If $\varphi$ is an embedding, then for small-enough $\alpha \in (0,1]$, the embedding $(x, i) \mapsto \varphi(\alpha \cdot x, i)$ is tiny.
\end{prop}
\begin{proof}
The map $(\varphi, x, i) \mapsto (\varphi(0, i), \varphi(x, i))$ is continuous.  If $\varphi$ is tiny, then $\{\varphi\} \times \mathbb{D}^l_n$ is contained in the preimage of the open set $L$; the first claim then follows from the tube lemma.  We address the second claim.  Since $(x, i) \mapsto (\varphi(0,i), \varphi(x,i))$ is continuous, the preimage of $L$ is an open neighborhood of the slice $\{0\} \times [n]$ in $\mathbb{D}^l_n$.  
For sufficiently-small $\alpha \in (0,1]$, the multiplication map $(x, i) \mapsto (\alpha \cdot x, i)$ lands in this preimage.
\end{proof}
\subsection{Building two useful continuous functions}
The next two results, Propositions \ref{prop:build_r} and \ref{prop:build_alpha} have the same style, analogous proofs, and will be used in \S\ref{sec:main_proof} to construct a homotopy inverse to $\zeta$.
\begin{prop} \label{prop:build_r}
There exists a continuous function $r \colon \PConf^n(M) \to \mathbb{R}_{>0}$ so that
\begin{itemize}
\item $r(p_1, \ldots, p_n) \cdot f(p_i)(x) \in U$ for all $(p_1, \ldots, p_n) \in \PConf^n(M)$ and $(x, i) \in \mathbb{D}^l_n$, and moreover,
\item $(x, i) \mapsto \exp(r(p_1, \ldots, p_n) \cdot f(p_i)(x))$ is always an element of $\Emb^f(\mathbb{D}^l_n, M)$.
\end{itemize}
\end{prop}
\begin{proof}
There are three stages of the proof.  First, we consider a single configuration $(p_1, \ldots, p_n) \in \PConf^n(M)$ and produce a value $r_0 >0$ so that any continuous function $r$ with $r(p_1, \ldots, p_n) \leq r_0$ satisfies both conditions at the configuration $(p_1, \ldots, p_n)$.  We call $r_0$ the ``suggested value'' at this configuration.  The second step shows that this same $r_0$ also works for nearby configurations in an open neighborhood of $(p_1, \ldots, p_n)$.  The last step uses paracompactness of $\PConf^n(M)$ and a partition of unity subordinate to this system of open neighborhoods to stitch a single continuous function.

\textbf{Constructing a suggested value for a single configuration.}  Let $(p_1, \ldots, p_n) \in \PConf^n(M)$, and recall that $U \subseteq TM$ denotes the open domain of the exponential function.  Since the function $(s, x, i) \mapsto s \cdot f(p_i)(x)$ is continuous, the preimage of $U$ is open, and,  containing the slice $\{0\} \times \mathbb{D}^l_n$, it also contains some product $[0, s_0) \times \mathbb{D}^l_n$ with $s_0 > 0$ by the tube lemma.  So the first condition is satisfied as long as we pick $r_0 \leq s_0$.

Write $\psi \colon \mathbb{D}^l_n \to M$ for the smooth immersion $\psi(x, i) = \exp(s_0 \cdot f(p_i)(x))$.  Since $\psi(0,i) = p_i$ and $i \neq j \implies p_i \neq p_j$, and since $D\psi(0, i)$ is a positive scalar multiple of $f(p_i)$ so that every convex combination $t \cdot D\psi(0,i) + (1-t)f(p_i)$ has full rank, Propositions \ref{prop:get_alpha} and \ref{prop:tiny_open} provide a number $\alpha \in (0,1]$ so that $(x,i) \mapsto \psi(\alpha \cdot x, i)$ is a tiny embedding that is adapted to $f$.  Since $f(p_i)$ is a linear map, $\exp(s_0 \cdot f(p_i)(\alpha \cdot x)) = \exp(s_0 \cdot \alpha \cdot f(p_i)(x))$, and so we may take $r_0 = s_0 \cdot \alpha$.

\textbf{The same $r_0$ works in an open neighborhood of its configuration.}  The function
$$
\begin{array}{ccc}
\PConf^n(M) \times \mathbb{D}^l_n & \longrightarrow & TM \\
(q_1, \ldots, q_n; \, x, i) & \longmapsto & r_0 \cdot f(q_i)(x).
\end{array}
$$
is continuous, and so the preimage of $U$ is open.  Since the slice $(p_1, \ldots, p_n) \times \mathbb{D}^l_n$ is contained in the preimage, the tube lemma supplies some open subset $V \subseteq \PConf^n(M)$ with $(p_1, \ldots, p_n) \in U$ so that
$$
(q_1, \ldots, q_n; \, x, i) \in V \times \mathbb{D}^l_n \;\;\;\; \implies \;\;\;\; r_0 \cdot f(q_i)(x) \in U.
$$
On this neighborhood, the formula $\exp(r_0 \cdot f(q_i)(x))$ makes sense, giving a continuous function
$$
\begin{array}{ccc}
V & \longrightarrow & C^\infty(\mathbb{D}^l_n, M) \\
(q_1, \ldots, q_n) & \longmapsto & ((x, i) \mapsto \exp(r_0 \cdot f(q_i)(x))).
\end{array}
$$
By Proposition \ref{prop:adapted_open}, the subset $\Emb^f(\mathbb{D}^l_n, M)$ is open, and so its preimage, which we call $V_0$, is open as well.  By construction, the same $r_0$ that works for $(p_1, \ldots, p_n)$ also works for all elements of $V_0$.

\textbf{Building a continuous function from the suggested values.}
For each configuration $(p_1, \ldots, p_n) \in \PConf^n(M)$, build the open set $V_0 = V_0(p_1, \ldots, p_n) \subseteq \PConf^n(M)$ as above, and its corresponding suggested value $r_0 = r_0(p_1, \ldots, p_n)$.  Since $(p_1, \ldots, p_n) \in V_0(p_1, \ldots, p_n)$, these sets form an open cover.  Using the paracompactness of $\PConf^n(M)$, choose a partition of unity subordinate to this cover.  Since the set of valid choices for $r_0$ at a configuration $(p_1, \ldots, p_n)$ is convex, we construct $r$ as the weighted average of its suggested values, where the partition of unity supplies the weights.
\end{proof}
\begin{prop} \label{prop:build_alpha}
There exists a continuous function $\alpha \colon \Emb^f(\mathbb{D}^l_n, M) \to \mathbb{R}_{>0}$ so that, for all $\varphi \in \Emb^f(\mathbb{D}^l_n, M)$, and setting $a = \alpha(\varphi)$, $(p_1, \ldots, p_n) = \zeta(\varphi)$, and $\psi = (\theta \circ \zeta)(\varphi)$, 
\begin{itemize}
\item $(p_i, \varphi(ax, i)) \in L$ and $(p_i, \psi(ax, i)) \in L$ for all $(x, i) \in \mathbb{D}^l_n$, and consequently
$$
t \cdot \log_{p_i} \varphi(ax, i) + (1-t) \cdot \log_{p_i} \psi(ax, i) \in U \mbox {  for all $t \in [0,1]$ by \eqref{equation:convex}, and }
$$
\item the smooth map $(t, x, i) \mapsto \Big(t,\; \exp\big( t \cdot \log_{p_i} \varphi(ax, i) + (1-t) \cdot \log_{p_i} \psi(ax, i) \big)\Big)$ is an element of the space $\Emb([0,1] \times \mathbb{D}^l_n, \, [0,1] \times M)$.
\end{itemize}
\end{prop}
\begin{proof}
There are three steps.  First, we consider a single $\varphi$ and produce a value $\alpha_0 >0$ satisfying both conditions at $\varphi$.  As before, we call $\alpha_0$ the ``suggested value'' for $\alpha$ at $\varphi$.  Both conditions will hold at $\varphi$ if our actual choice $\alpha(\varphi)$ has $\alpha(\varphi) \leq \alpha_0$.  The second step shows that this same $\alpha_0$ works in an open neighborhood of $\varphi$, and not just for $\varphi$ itself.  Finally, covering $\Emb^f(\mathbb{D}^l_n, M)$ by such neighborhoods, we use a partition-of-unity argument to build $\alpha$ as a weighted average of its suggested values.

\textbf{Constructing a suggested value for a single embedding.}  The continuous function
$$
(x, i, \beta) \mapsto (\varphi(0,i), \varphi(\beta \cdot x);\, \psi(0,i), \psi(\beta \cdot x))
$$
sends the slice $\mathbb{D}^l_n \times \{0\}$ into the open set $L \times L$, and hence, by the tube lemma, there exists $\beta_0 >0$ so that the first condition is satisfied as long as we choose $\alpha_0 \leq \beta_0$.

For each $t \in [0,1]$, define $\gamma_t \in C^\infty(\mathbb{D}^l_n, M)$ by the formula
$$
\gamma_t(x, i) = \exp_{p_i}\big( t \cdot \log_{p_i} \varphi(\beta_0x, i) + (1-t) \cdot \log_{p_i} \psi(\beta_0  x, i) \big).
$$
Since, for any $p \in M$, we have $D\exp_p(0) = 1_{T_pM}$ and similarly, $D\log_p(p) = 1_{T_pM}$,
\begin{align*}
D\gamma_t(0, i) &= t \cdot D\varphi(0, i) + (1-t) \cdot D \psi(0, i) \\
&= t \cdot D\varphi(0, i) + (1-t) f(\varphi(0, i)),
\end{align*}
and hence $D\gamma_t(0, i)$ has full rank because $\varphi$ is adapted to $f$.  It follows that the graph of $t \mapsto \gamma_t$,
$$
\begin{array}{ccc}
\Gamma \colon [0,1] \times \mathbb{D}^l_n & \longrightarrow & [0,1] \times M \\
(t, x, i) & \longmapsto & (t, \gamma_t(x, i)),
\end{array}
$$
also has full-rank derivative along the compact subset $K = [0,1] \times \{0\} \times [n]$.  Moreover, $\Gamma_t$ restricts to an injection along the subset because $\gamma_t(0,i) = p_i = \varphi(0,i)$ and $\varphi$ is injective.  Applying Lemma \ref{lem:embedding_neighborhood}, there exists a small open neighborhood of the subset where $\Gamma$ is an embedding.  By the tube lemma, this neighborhood contains all small-enough values of $x \in \mathbb{D}^l$, and so we may pick $\alpha_0$ accordingly.  The map in the second condition is then an embedding because it factors through this restriction of $\Gamma$.

\textbf{The same $\alpha_0$ works in an open neighborhood of $\varphi$.}
Writing $\psi_{\eta} = (\theta \circ \zeta)(\eta)$, define a continuous map
$$
\begin{array}{ccc}
\Emb^f(\mathbb{D}^l_n, M) \times \mathbb{D}^l_n & \longrightarrow & M \times M \times M \times M \\
(\eta, x, i) & \longmapsto & (\eta(0,i), \eta(\alpha_0 \cdot x);\, \psi_{\eta}(0,i), \psi_{\eta}(\alpha_0 \cdot x)).
\end{array}
$$
Since $\psi_{\varphi} = \psi$, the slice $\{\varphi \} \times \mathbb{D}^l_n$ is contained in the preimage of the open set $L \times L$, and so we use the tube lemma to find an open set $V \subseteq \Emb^f(\mathbb{D}^l_n, M)$ containing $\varphi$ for which the expression
$$
t \cdot \log_{p_i} \eta(\alpha_0x, i) + (1-t) \cdot \log_{p_i} \psi_\eta(\alpha_0x, i)
$$
makes sense, and hence lives in $U$ by \eqref{equation:convex}.  Define the continuous map
$$
\begin{array}{ccc}
V \times \mathbb{D}^l_n & \longrightarrow & C^{\infty}([0,1] \times \mathbb{D}^l_n, [0,1] \times M) \\
(\eta, x, i) & \longmapsto & \Bigg( (t, x, i) \mapsto \Big(t,\; \exp\big( t \cdot \log_{p_i} \eta(\alpha_0x, i) + (1-t) \cdot \log_{p_i} \psi_\eta(\alpha_0x, i) \big)\Big) \Bigg).
\end{array}
$$
The slice $\{\varphi\} \times \mathbb{D}^l_n$ is contained in the preimage of the set $\Emb^f(\mathbb{D}^l_n, M)$, which is open by Proposition~\ref{prop:adapted_open}.  The tube lemma then supplies the required open neighborhood.

\textbf{Building a continuous function from the suggested values.}  The space of smooth maps $C^{\infty}(\mathbb{D}^l_n, M)$ is paracompact since it is metrizable \cite[p. 62]{Hirsch}, and so the same is true of the subspace $\Emb^f(\mathbb{D}^l_n, M)$.  If some value $a > 0$ satisfies both conditions, then the same is true for any $a'$ in the convex set $(0, a]$.  Consequently, a partition of unity subordinate to the open cover $\{V_{\varphi}\}$ found in the previous step provides a continuous choice of convex combination, and hence a global choice of $\alpha$.
\end{proof}

\subsection{A homotopy inverse to $\zeta$} \label{sec:main_proof}
Use Propositions \ref{prop:build_r} and \ref{prop:build_alpha} to build continuous functions $r$ and $\alpha$, and define $\theta \colon \PConf^n(M) \to \Emb^f(\mathbb{D}^l, M)$ by the formula
$$
\theta(p_1, \ldots, p_n)(x, i) = \exp(r(p_1, \ldots, p_n) \cdot f(p_i)(x)).
$$

\begin{thm} \label{thm:equiv}
The function $\theta$ is homotopy inverse to $\zeta \colon \Emb^f(\mathbb{D}^l_n, M) \to \PConf^n(M)$.
\end{thm}
\begin{proof}
Since $\zeta \circ \theta = 1$, it is enough to build a homotopy from $\theta \circ \zeta$ to the identity.  We build this homotopy as a composite of three smaller homotopies, with intermediate stops at
$$
\begin{array}{ccc}
\varphi \mapsto ((x, i) \mapsto \psi(\alpha x, i)) & \mbox{ and } &  \varphi \mapsto ((x, i) \mapsto \varphi(\alpha x, i)),
\end{array} 
$$
recalling that the variables $\psi = (\theta \circ \zeta)(\varphi)$ and $\alpha = \alpha(\varphi)$ implicitly depend on $\varphi$.  The outer two homotopies come from any path connecting the function $\alpha$ to the constant function $1$, perhaps $t \mapsto t \cdot \alpha + (1-t) \cdot 1$.  The middle homotopy uses that the smooth map
$$
(t, x, i) \mapsto \Big(t,\; \exp\big( t \cdot \log_{p_i} \varphi(\alpha x, i) + (1-t) \cdot \log_{p_i} \psi(\alpha x, i) \big)\Big)
$$
is known to be an embedding by the construction of $\alpha$ in Proposition \ref{prop:build_alpha}.  In particular, for each $t \in [0,1]$, the map
$$
(x, i) \mapsto \exp\big( t \cdot \log_{p_i} \varphi(\alpha x, i) + (1-t) \cdot \log_{p_i} \psi(\alpha x, i) \big)
$$
is an embedding.  Since $\exp$ and $\log$ are inverse, it is immediate that this formula provides a homotopy-through-embeddings connecting the two intermediate stops.
\end{proof}
\section{Proof of Theorem \ref{thm:main}} \label{sect:mainproof}
Let $M$ be a Riemannian manifold equipped with a smoothly-varying $l$-frame $f(p)$ for $p \in M$, as in \S\ref{sec:adapted}.  Theorem \ref{thm:equiv} provides a homotopy equivalence 
$$\Emb^f(\mathbb{D}^l_n, M) \simeq \PConf^n(M),$$
so $\Emb^f(\mathbb{D}^l_n, M)$, which was introduced in Definition \ref{defn:adapted}, is a reasonable model for configuration space.  The category $\mathcal{F}E_l^{op}$ acts on $\Emb^f(\mathbb{D}^l_n, M)$ by precomposition, as in Proposition \ref{prop:action}.  We illustrate precomposition with an embedding of four disjoint disks into three disjoint disks in the case $M = S^1 \times S^1$:\\
\begin{center}
\scalebox{.65}{\includegraphics{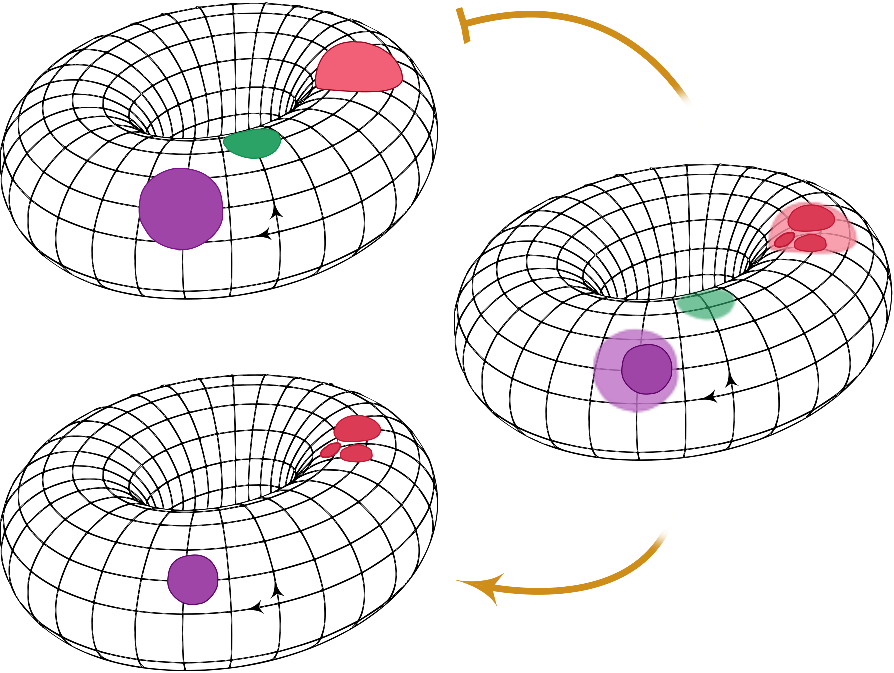}}\\
\end{center}

It follows that $[n] \mapsto H^i(\Emb^f(\mathbb{D}^l_n, M)) \simeq H^i(\PConf^n(M))$ is an $\mathcal{F}E_l$-representation.  Moreover, since homotopic maps induce the same linear map on cohomology, this action descends to $\pi_0 \mathcal{F}E_l$, the quotient category induced by taking path components of every hom-space. We have mostly proved the following theorem, which is stronger than Theorem~\ref{thm:main}.
\begin{thm}\label{thm:refinedmain}  Let $M$ be a smooth manifold admitting $l$ everywhere linearly independent vector fields, and let $k$ be a coefficient field.  The assignment $[n] \mapsto H^i(\PConf^n(M); k)$ is a representation of $\pi_0 \mathcal{F}E_l$, a category that simplifies further depending on $l$:
$$
\pi_0 \mathcal{F}E_l = \left\{
     \begin{array}{lr}
       \mathcal{FI} & : l = 0 \phantom{.}\\
       \ncFin & : l = 1 \phantom{.}\\
       \Fin & : l \geq 2.
     \end{array}
   \right.
$$
Further, if $M$ is finite type, connected, orientable, and dimension at least two, then this representation is finitely generated.
\end{thm}
\begin{proof}
The construction explained above gives the action of $\pi_0 \mathcal{F}E_l$; we must show that these representations are finitely generated.  The paper \cite{CEF} deduces finite generation even for the case $l=0$ from the spectral sequence of Totaro \cite{Totaro}, from which it follows for every other $l$ (since $\pi_0\mathcal{F}E_0$ is naturally a subcategory of every other $\pi_0\mathcal{F}E_l$.)

The simplification of the category $\pi_0 \mathcal{F}E_l$ depends on the following easy observations.  When $l=0$, embeddings of $0$-balls are just injections among disjoint unions of points; it follows that $\pi_0\mathcal{F}E_0 = \mathcal{FI}$, the category of finite sets with injections.  For $l=1$, a configuration of line segments in a line segment (up to homotopy) is the same as an ordering of the segments; this provides the isomorphism $\pi_0\mathcal{F}E_1 \simeq \mathcal{N}$.  When $l \geq 2$, $\mathcal{F}E_l([n], [1])$ is a connected space and so $\pi_0 \mathcal{F}E_l = \mathcal{F}$.
\end{proof}

\begin{rem}
The notation $\mathcal{F}E_l$ is meant to evoke the ``category of operations'' for the little $l$-balls operad $E_l$, introduced by May-Thomason in \cite{maythomason}.  (Another popular notation for this category is $E_l^{\otimes}$.)  Our category is essentially an unbased analog of their category.
\end{rem}

\begin{rem}\label{rem:cosimplicialdimensions}
A version of this construction for $l=1$ is essentially present in work of Berrick, Cohen, Wong, and Wu \cite{BCWW}, \cite{WuSimplicial}, although they are content to use the subcategory $\Delta \subseteq \mathcal{N}$.  It must be noted that Corollary~\ref{co:polydim}, the statement on dimensions of cohomology of configuration spaces of manifolds with flows, could be deduced immediately from \cite{BCWW} and the Dold-Kan representation Lemma \ref{lem:dkrl}.  Nevertheless, such a proof seems not to be present in the literature, despite considerable interest in the asymptotic behavior of these dimension sequences; and the use of $\mathcal{N}$ as opposed to $\Delta$ is necessary if one wishes to obtain stability theorems for the character of the symmetric group actions, as in Corollary~\ref{co:polychar}.
\end{rem}
\bibliographystyle{alphaurl}
\bibliography{references}

\begin{thebibliography}{BCWW06}

\bibitem[AH72]{vectorfields}
Paul Alexandroff and Heinz Hopf.
\newblock {\em Topologie. {B}and {I}}.
\newblock Chelsea Publishing Co., Bronx, N. Y., first edition, 1972.
\newblock Grundbegriffe der mengentheoretischen Topologie, Topologie der
  Komplexe, topologische Invarianzs{\"a}tze und anschliessende
  Begriffsbildungen. Verschlingungen im $n$-dimensionalen euklidischen Raum,
  stetige abbildungen von Polyedern, Including an addendum based on an article
  by Hing Tong, George Kozlowski and Mary Powderly (``On a problem of
  Alexandroff and Hopf'', Bull. Inst. Math. Acad. Sinica {{\bf{3}}} (1975), no.
  1, 15--16).

\bibitem[Arn69]{ArnoldClassical}
V.I. Arnol'd.
\newblock The cohomology ring of the colored braid group.
\newblock {\em Mathematical notes of the Academy of Sciences of the USSR},
  5(2):138--140, 1969.

\bibitem[BCWW06]{BCWW}
A.~J. Berrick, F.~R. Cohen, Y.~L. Wong, and J.~Wu.
\newblock Configurations, braids, and homotopy groups.
\newblock {\em J. Amer. Math. Soc.}, 19(2):265--326, 2006.

\bibitem[CEF15]{CEF}
Thomas Church, Jordan~S. Ellenberg, and Benson Farb.
\newblock {FI}-modules and stability for representations of symmetric groups.
\newblock {\em Duke Math. J.}, 164(9):1833--1910, 2015.

\bibitem[CEFN14]{CEFN}
Thomas Church, Jordan~S. Ellenberg, Benson Farb, and Rohit Nagpal.
\newblock F{I}-modules over {N}oetherian rings.
\newblock {\em Geom. Topol.}, 18(5):2951--2984, 2014.

\bibitem[CF13]{ChurchFarbRTHS}
Thomas Church and Benson Farb.
\newblock Representation theory and homological stability.
\newblock {\em Adv. Math.}, 245:250--314, 2013.

\bibitem[Chu12]{Church}
Thomas Church.
\newblock Homological stability for configuration spaces of manifolds.
\newblock {\em Invent. Math.}, 188(2):465--504, 2012.

\bibitem[Coh88]{CohenClassical}
F.~R. Cohen.
\newblock Artin's braid groups, classical homotopy theory, and sundry other
  curiosities.
\newblock In {\em Braids ({S}anta {C}ruz, {CA}, 1986)}, volume~78 of {\em
  Contemp. Math.}, pages 167--206. Amer. Math. Soc., Providence, RI, 1988.

\bibitem[CP15]{canteropalmer}
Federico Cantero and Martin Palmer.
\newblock On homological stability for configuration spaces on closed
  background manifolds.
\newblock {\em Doc. Math.}, 20:753--805, 2015.

\bibitem[CT78]{CT}
F.~R. Cohen and L.~R. Taylor.
\newblock Configuration spaces: applications to {G}elfand-{F}uks cohomology.
\newblock {\em Bull. Amer. Math. Soc.}, 84(1):134--136, 1978.

\bibitem[Dol58]{DoldClassical}
Albrecht Dold.
\newblock Homology of symmetric products and other functors of complexes.
\newblock {\em Ann. of Math. (2)}, 68:54--80, 1958.

\bibitem[Hir94]{Hirsch}
Morris~W. Hirsch.
\newblock {\em Differential topology}, volume~33 of {\em Graduate Texts in
  Mathematics}.
\newblock Springer-Verlag, New York, 1994.
\newblock Corrected reprint of the 1976 original.

\bibitem[Knu17]{knudsen}
Ben Knudsen.
\newblock Betti numbers and stability for configuration spaces via
  factorization homology.
\newblock {\em Algebr. Geom. Topol.}, 17(5):3137--3187, 2017.
\newblock \href {http://dx.doi.org/10.2140/agt.2017.17.3137}
  {\path{doi:10.2140/agt.2017.17.3137}}.

\bibitem[McD75]{McDuffClassical}
Dusa McDuff.
\newblock Configuration spaces of positive and negative particles.
\newblock {\em Topology}, 14:91--107, 1975.

\bibitem[MT78]{maythomason}
J.~P. May and R.~Thomason.
\newblock The uniqueness of infinite loop space machines.
\newblock {\em Topology}, 17(3):205--224, 1978.

\bibitem[PR02]{HochFunctorHomology}
T.~Pirashvili and B.~Richter.
\newblock Hochschild and cyclic homology via functor homology.
\newblock {\em $K$-Theory}, 25(1):39--49, 2002.

\bibitem[Put96]{fulltransformationmonoid}
Mohan~S. Putcha.
\newblock Complex representations of finite monoids.
\newblock {\em Proc. London Math. Soc. (3)}, 73(3):623--641, 1996.

\bibitem[Rai09]{rains}
Eric~M. Rains.
\newblock The action of {$S_n$} on the cohomology of {$\overline M_{0,n}(\Bbb
  R)$}.
\newblock {\em Selecta Math. (N.S.)}, 15(1):171--188, 2009.

\bibitem[RW13]{RWStability}
Oscar Randal-Williams.
\newblock Homological stability for unordered configuration spaces.
\newblock {\em Q. J. Math.}, 64(1):303--326, 2013.

\bibitem[Seg74]{SegalClassical}
Graeme Segal.
\newblock Categories and cohomology theories.
\newblock {\em Topology}, 13:293--312, 1974.

\bibitem[Ter]{MSE1782385}
Pedro~S{\'a}nchez Terraf.
\newblock Continuous function injective over a compact set, and locally
  injective on each point of the set.
\newblock Mathematics Stack Exchange.
\newblock version: 2016-05-12.
\newblock URL: \url{https://math.stackexchange.com/q/1782385}.

\bibitem[Tot96]{Totaro}
Burt Totaro.
\newblock Configuration spaces of algebraic varieties.
\newblock {\em Topology}, 35(4):1057--1067, 1996.

\bibitem[Wei94]{weibel}
Charles~A. Weibel.
\newblock {\em An introduction to homological algebra}, volume~38 of {\em
  Cambridge Studies in Advanced Mathematics}.
\newblock Cambridge University Press, Cambridge, 1994.

\bibitem[WG14]{UniformlyPresentedVectorSpaces}
John~D. Wiltshire-Gordon.
\newblock Uniformly presented vector spaces, 2014.
\newblock \href {http://arxiv.org/abs/1406.0786} {\path{arXiv:1406.0786}}.

\bibitem[Wu10]{WuSimplicial}
Jie Wu.
\newblock Simplicial objects and homotopy groups.
\newblock In {\em Braids}, volume~19 of {\em Lect. Notes Ser. Inst. Math. Sci.
  Natl. Univ. Singap.}, pages 31--181. World Sci. Publ., Hackensack, NJ, 2010.

\end{thebibliography}
\end{document}